\documentclass[11pt,amscd,amssymb,verbatim]{amsart}
\usepackage[mathscr]{euscript}
\usepackage{epsfig}
\usepackage{graphics}
\usepackage{amsbsy}

\numberwithin{equation}{section}
\newtheorem{theorem}{Theorem}
\newtheorem{lemma}{Lemma}
\newtheorem{corollary}{Corollary}
\newtheorem{remark}{Remark}
\newtheorem{definition}{Definition}
\newtheorem{proposition}{Proposition}

\begin{document}
\title[  Uniform  limit theorems under random truncation]
 { uniform  limit theorems under random truncation}

\author{ V. Fakoor$^{1}$ \quad R. Zamini$^{2}$ }

\email{fakoor@math.um.ac.ir}
\email{rahelehzamini@yahoo.com}

\keywords{  Left-truncation,  Lynden-Bell estimator, Uniform limit theorems}

\maketitle

\begin{center}
{\it $^1$Department of Statistics, Faculty  of Mathematical Sciences,\\
Ferdowsi University of Mashhad, Iran.}\\
{\it $^2$ Department of   Mathematics and  Computer Science, Kharazmi  University, Tehran,
 Iran.\\}
\end{center}

\begin{abstract}
In this paper we study  uniform versions of  two limit theorems in random left truncation model (RLTM). The law of large numbers (LLN) and   the central limit theorem (CLT) have been obtained under the bracketing entropy conditions in this setting. The uniform LLN and the uniform CLT of the present paper extend the one dimensional LLN and the one dimensional CLT under RLTM respectively.
\end{abstract}

\section{\bf Introduction and Preliminaries}
Limit theorems have always been of central importance in probability theory.
Two most important  of  limits theorems  constitute: the  LLN and  the CLT. These topics have their own importance in the classical probability
theory as well applications in  statistical inference. In
recent years, under the topic of empirical process theory, authors have become interested
in the uniform analogue of the two  theorems: the uniform LLN of
Glivenko-Cantelli type and  the uniform CLT for Donsker type.

Let $X_{1},\ldots ,X_{n}$ be a sequence of random variables defined  on a probability space
$(\Omega, \mathcal{T}, P)$  with common law  distribution   $\mathbb{P}$ on $\mathbb{R}$. We denote by
$\mathbb{P}_n=n^{-1}\sum_{i=1}^n {\delta}_(X_i)$ the usual empirical measure, where ${\delta}_x$ is a Kronecker delta.
Let $\mathcal{F}$
be a set of measurable  real valued functions on $\mathbb{R}$.
The unifom version of the LLN states that
\begin{eqnarray}\label{ullw}
\sup_{\varphi \in \mathcal{F} }(\mathbb{P}_n-\mathbb{P})\varphi\rightarrow 0 \quad a.s.,
\end{eqnarray}
where $\mathbb{P}\varphi=\int \varphi d\mathbb{P}$.
A class $\mathcal{F}$,  for which \eqref{ullw}       holds,   is called a \textit{Glivenko--Cantelli class}.
DeHardt [\ref{De}] obtained the uniform LLN  for the sequence
of independent and identically distributed (i.i.d.)  random variables  under bracketing entropy (see for example Van der  Vaart and Wellner [\ref{Van}]).
DeHardt's result states that if $\mathcal{F}$ has a bracketing entropy then \eqref{ullw}
is obtained.

The $\mathcal{F}-$indexed empirical process is  given by
\begin{eqnarray}\label{ulw}
\varphi \mapsto \sqrt{n}(\mathbb{P}_n-\mathbb{P})\varphi \quad   \varphi \in  \mathcal{F}.
\end{eqnarray}
This process  can be viewed as a map into ${\ell}^{\infty}(\mathcal{F})$,
 where  ${\ell}^{\infty}(\mathcal{F})$ denotes the Banach space of bounded real-valued functions $\psi$ on $\mathcal{F}$, normed by
 $||\psi||_{\mathcal{F}}:=\sup_{\varphi \in \mathcal{F} } |\psi(\varphi)|$.  The processes $\varphi  \mapsto \sqrt{n}(\mathbb{P}_n-\mathbb{P})\varphi, \varphi \in  \mathcal{F}$
 converge in law to a Gaussian process in  ${\ell}^{\infty}(\mathcal{F})$, called the $\mathbb{P}-$Brownian bridge indexed by
$ \mathcal{F}.$  A  class $\mathcal{F}$   for which this is the true,   is named  a  $\mathbb{P}-$Donsker class ( see Van der Vaart and Wellner [\ref{Van}]).

The study of the asymptotic behaviors of  the  process \eqref{ullw} is a central topic in empirical process
theory, and it is well known that this behavior depends on the complexity or  "entropy" of $\mathcal{F}$.

 Ossiander [\ref{Oss}] developed
the uniform CLT for the sequence of i.i.d.
 random variables under certain metric integrability conditions.  The Ossiander result states that under an integrability condition on the metric entropy with bracketing in $L^2(\mathbb{P})$, $\sqrt{n}(\mathbb{P}_n-\mathbb{P})$,  as random elements in ${\ell}^{\infty}(\mathcal{F})$,
 converges in law to a mean zero Gaussian process  $\{W(\varphi) : \varphi \in \mathcal{F}\}$   with the covariance structure $EW({\varphi}_1)W({\varphi}_2) = E{\varphi}_1(X){\varphi}_2(X)$.

There is  lack of works on uniform limit theorems  for a  RLTM in the literature.
 The goal of the paper is to prove two fundamental theorems  of uniform LLN and  CLT . Here we present  the uniform LLN  of
Glivenko-Cantelli type and  the uniform CLT  for Donsker type
 in the  RLTM which postpone to section 2.

Consider that there is a finite population $\mathcal{P}$ whose size is large, deterministic and is denoted by $N$. Each element of $\mathcal{P}$ contains two independent  random variables denoted by $Y$ and $T$ with  distribution functions, respectively shown by $F$ and $G$.
$Y$ is the variable of interest and $T$ is the left truncation random variable.
So until now we have $N$ i.i.d.  random variables $\{(Y_i,T_i); 1\leq i \leq N \}$. Suppose that  $(Y,T)$ is observed if $Y \geq T$, otherwise we have no information about them. We show the observable random variables by $\{(Y_i,T_i); 1\leq i \leq n \}$ with $Y_i \geq T_i$.
 As a direct result $n$ is a binomial random variable, with sample size $N$ and success probability $\alpha  = P \left( Y \ge T \right)$.

  The conditional distribution of $(Y,T)$ given $Y\geq T$ is denoted by $H^{*}$,
\begin{eqnarray}
{H^*}\left( {y,t} \right) &=& P \left( {Y \le y,T \le t}| Y \geq T \right) \nonumber\\
 &=& {\alpha ^{ - 1}}\int_{ - \infty }^y {G\left( {t \wedge u} \right)dF\left( u \right)},
\end{eqnarray}
Marginal distribution functions of $Y$ and $T$ are given by
\begin{eqnarray}
{F^*}\left( y \right) = {H^*}\left( {y,\infty } \right) = {\alpha ^{ - 1}}\int_{ - \infty }^y {G\left( u \right)dF\left( u \right)}, \nonumber
\end{eqnarray}
and
\begin{eqnarray}
G^*(t) = H^*(\infty,t) = \alpha ^{ - 1} \int_{ - \infty }^t (1-F(u))d G(u). \nonumber
\end{eqnarray}
The corresponding empirical distributions are defined  by
\begin{eqnarray}
F_n^*(y) = \frac{1}{n} \sum_{i=1}^n I(Y_i\leq y), \quad  G_n^*(t) = \frac{1}{n} \sum_{i=1}^n I(T_i\leq t), \nonumber
\end{eqnarray}
where $I(A)$ is the indicator function of $A$.

Random truncation restricts the observation range of $X$ and $Y.$ Only $F_0(x)=P(X \leq x |X \geq a_G)$
and $G_0(y)=P(Y \leq y |Y \leq b_F)$ can be estimated, where
\begin{eqnarray*}
a_{F}:=\inf \{x:\ F(x)>0\}\quad   and\quad  b_{F}:=\sup \{x:\ F(x)<1\}
\end{eqnarray*}
are the lower and upper boundaries of the support of the distribution of $X$. Let $a_G$  and $b_G$  be similarly defined.
If $a_G \leq a_F$, then $F=F_0$.

 The nonparametric maximum likelihood estimator of $F$ was first derived by   Lynden-Bell [\ref{Lynden-Bell}], which we refer to it by $F_n(\cdot)$.
If there are no ties in the data, it is given by
\begin{eqnarray}\label{lynden estimators}
{F_n}\left( y \right) = 1 - \prod\limits_{i:{Y_i} \le y} {\left[ {\frac{{n{C_n}\left( {{Y_i}} \right) - 1}}{{n{C_n}\left( {{Y_i}} \right)}}} \right]}, \,\,\,
\end{eqnarray}
in which
 \begin{eqnarray*}{C_n}( y ) = G_n^*( y ) - F_n^*\left( {y^- } \right) = \frac{1}{n}\sum\limits_{i = 1}^n {{I_{\left\{ {{T_i} \le y \le {Y_i}} \right\}}}} , \,\, y \in \mathbb{R},
 \end{eqnarray*}
where is the empirical estimator of
\begin{eqnarray}\label{C(y)}
C( y ): = {G^*}( y ) - {F^*}\left( y \right) = {\alpha ^{ - 1}}G\left( y \right)\left( {1 - F\left( y^- \right)} \right),\,\,
y \in \mathbb{R}.
\end{eqnarray}
Left limit $\lim_{y\uparrow s} g(y)$ is denoted by $g(s^-)$.

The following definitions, are conveniently collected in  Van der
Vaart and Wellner [\ref{Van}], so we follow their notation. First, we define entropy with bracketing. Let  $(\mathcal{F}, ||\cdot||_p)$
be a subset of a normed space of $(L^p(\mathbb{P}),||\cdot||_p )$ where
\begin{eqnarray*}
 (||\varphi||_p)^p=\int |\varphi|^p d\mathbb{P}.
\end{eqnarray*}
 \begin{definition} Given two functions $l$ and $u$ in $L^p(\mathbb{P})$, the bracket  $[l,u]$ is the set of all functions $\varphi$ with $l \leq \varphi \leq u.$
 An $\epsilon$-bracket in $L^p(\mathbb{P})$ is a bracket  $[l,u]$ with  $||u-l ||_p<\epsilon$. The bracketing number
 $N_{[ ~]}(\epsilon,\mathcal{F}, ||\cdot ||_p )$ is the minimum number of $\epsilon$-brackets needed to cover
  $\mathcal{F}$. The entropy with bracketing is the logarithm of the bracketing number.
 \end{definition}
  The notion of entropy with bracketing has been introduced by Dudley [\ref{Dud}]  and the importance of $L^2(\mathbb{P})$-entropy with bracketing has
  been pointed out by Ossiander [\ref{Oss}]. In the rest of the paper, whenever unambiguous we write $N_{[ ~]}(\epsilon)$ instead of $N_{[ ~]}(\epsilon,\mathcal{F}, ||\cdot ||_p ).$
 \begin{definition}
The bracketing entropy integral of a class
of functions $\mathcal{F}$  is defined by
\begin{eqnarray}
  J(\delta):= J_{[ ~]}(\delta, \mathcal{F}, ||\cdot ||_p ) =\int_0^\delta \sqrt{\log N_{[ ~]}(\epsilon,\mathcal{F}, ||\cdot ||_p ) }  d\epsilon,
\end{eqnarray}
where  $N_{[~ ]}(\cdot)$ are the bracketing numbers of $\mathcal{F}$ with
respect to the norm $|| \cdot ||_p.$
\end{definition}
\begin{definition}
 A measurable function $\Phi $ is an envelope function for $\mathcal{F}$ if $|\varphi(x)|\leq \Phi(x)$ for all $\varphi \in \mathcal{F}$ and all $x$. If
 $\sup_{\varphi \in \mathcal{F}}|\varphi(x)|$ is  measurable, it will be called the envelope function of $\mathcal{F}$.
\end{definition}
\begin{definition}
A sequence of ${\ell}^{\infty}(\mathcal{F})$-valued random functions $\{Y_n\}$  converges almost surely
to a constant $c$ if $P^*(\sup_{\varphi \in \mathcal{F} }Y_n(\varphi) \rightarrow c)=P(\sup_{\varphi \in \mathcal{F} }Y_n(\varphi)^* \rightarrow c)=1$.
Here $P^*$ denotes the outer probability, and $\sup_{\varphi \in \mathcal{F} }Y_n(\varphi)^* $ is the measurable cover
function of $\sup_{\varphi \in \mathcal{F} }Y_n(\varphi) $.
\end{definition}
We use the following definition of weak convergence which is originally due to Hoffman-J\"{o}rgensen [\ref{Hoff}].
\begin{definition} A sequence of   ${\ell}^{\infty}(\mathcal{F})$-valued random functions $\{Y_n\}$  converges in law to a ${\ell}^{\infty}(\mathcal{F})$-valued random function $\{Y\}$  whose law concentrates on a separable subset of ${\ell}^{\infty}(\mathcal{F})$ if
 \begin{eqnarray*}Eg(Y)=\lim_{n\rightarrow\infty}E^*g(Y_n)   \quad  \forall g \in U({\ell}^{\infty}(\mathcal{F}), ||\cdot||_{\mathcal{F}}),
  \end{eqnarray*}
where  $U({\ell}^{\infty}(\mathcal{F}), ||\cdot||_{\mathcal{F}})$ is the set of  all bounded, uniformly continuous function from
 $({\ell}^{\infty}(\mathcal{F}), ||\cdot||_{\mathcal{F}})$  into $\mathbb{R}.$
 Here $E^*$ denotes the upper expectation with respect to the outer probability $P^*.$ We denote this convergence
  by $Y_n\Rightarrow Y.$
\end{definition}

The layout of this paper is as follows.  In Section 2, we obtain our main theorems and results. In order to prove the main theorems we need the some auxiliary results, which are contained in Section 3.

\section{\bf Main results}


\subsection{\bf Uniform LLN}
 In this subsection, we assume that  $\mathcal{F} $ is  a class of real-valued measurable functions on $\mathbb{R}$ such that  $\mathcal{F} \subseteq L^1(F):= \{\varphi: \int |\varphi(x)| F(dx)< \infty \}.$
 The below assumption is imposed throughout this subsection to achieve Glivenko-Cantelli type theorem
for $W_n(\varphi):=\int \varphi d( F_n-F).$ \\

{\bf Assumption A} $F$ and $G$ are continuous with $a_G < b_F.$

\begin{theorem}\label{ulln}
Let $N_{[~]}(\epsilon, \mathcal{F}, L^1(F))<\infty,$ for every $\epsilon>0$, then under Assumption $\bold A$ we have
 \begin{eqnarray}
 \sup_{\varphi \in \mathcal{F}} \left|  \int \varphi d(F_n-F) \right| \longrightarrow 0 \quad a.s.
\end{eqnarray}
as $n\rightarrow\infty$.
\end{theorem}
 In order to start the proof of Theorem 1, we shall state the following lemma.

\begin{lemma}\label{Hhe}Under $\bold A,$  for any measurable  $\varphi$,  one can write
\[ \lim_{n\rightarrow\infty} \int \varphi dF_n=\int \varphi dF \quad a.s. \]
\end{lemma}
\begin{proof}
 By using relation of $\varphi=\varphi^+-\varphi^-$ in Theorem 4.3. of  He and Yang [\ref{He}] we can obtain the result.
\end{proof}

\begin{proof}[\textbf{Proof of Theorem \ref{ulln}}]
Let $ \epsilon >0$ is given.
By definition of the bracketing number, we can find  finitely many $\epsilon$-brackets $[l_i,u_i]$ whose union contains $\mathcal{F}$ and such that
$\int (u_i - l_i )d F<\epsilon$ for every $i=1,\ldots,N_{[~ ] }(\epsilon) $.
Then, for every $\varphi \in  \mathcal{F},$ there is bracket such that
\begin{eqnarray*}\label{}
W_n(\varphi) &=& \int \varphi dF_n-\int \varphi F\\
&\leq &  \int u_i d(F_n - F) + \int (u_i - l_i )d F \\
\end{eqnarray*}
Therefore,
\[
\sup_{\varphi \in \mathcal{F}}W_n(\varphi)  \leq  \max_{1\leq i \leq N_{[~] }(\epsilon)}  \int u_i d(F_n - F) +\epsilon
\]
Thus, by Lemma \ref{Hhe}
\[\limsup_{n \rightarrow \infty }  \sup_{\varphi \in \mathcal{F}}W_n(\varphi)  \leq  \epsilon \quad a.s. \]
Combination with a similar argument for $\inf_{\varphi \in \mathcal{F}}W_n(\varphi)$ yields that
$$\limsup_{n\rightarrow\infty} \sup_{\varphi \in \mathcal{F}}|W_n(\varphi)|^*\leq\epsilon  \quad a.s.$$
for every $\epsilon >0$.  Take a sequence $\epsilon_m\downarrow 0$ to see that
$\limsup$ must actually be zero almost surely.
\end{proof}
 \begin{corollary} Let ${\varphi}_0 : \mathbb{R} \rightarrow \mathbb{R}$ be a measurable function such that $\int |{\varphi}_0 | dF < \infty. $ Then,
 \begin{eqnarray}
   \sup_{t \in R} \left|  \int_{-\infty }^t  {\varphi}_0 (x) (F_n-F)(dx) \right| \longrightarrow 0 \quad a.s.
\end{eqnarray}
as $n\rightarrow\infty$.
\end{corollary}
\begin{proof}
Apply Theorem \ref{ulln} to $\mathcal{F}=\{{\varphi}_0.I(-\infty,t]:t \in \mathbb{R}\}$.
\end{proof}

\subsection{\bf Uniform CLT}

In this subsection we give a CLT for $\mathcal{F}$-indexed empirical processes
$$G_n(\varphi):=\sqrt{n} \int \varphi d(F_n-F), \quad  \varphi\in\mathcal{F}.$$
We assume that  $\mathcal{F} $ is  a class of real-valued measurable functions on $\mathbb{R}$ such that  $\mathcal{F} \subseteq L^2(F):= \{\varphi: \int {\varphi}^2(x) F(dx)< \infty \}.$
 The below assumption is imposed throughout this subsection to achieve Donsker type theorem
for $G_n(\varphi).$ \\


{\bf Assumption B.} $F$ is continuous with $a_G < a_F.$ \\


\begin{theorem}\label{uclt}
Let $\mathcal{F}$ be a class of functions with $J(1)<\infty$, satisfying Assumption
$\textbf{B}$. Then $G_n \Rightarrow W$ as elements of $B(\mathcal{F})$ where $\{W(\varphi): \varphi\in \mathcal{F}\}$ is a Gaussian process with the mean $EW(\varphi)=0$ and the covariance function is given by
\[Cov(W({\varphi}_1),W({\varphi}_2))=Cov(\zeta({\varphi}_1),\zeta({\varphi}_2)),\]
where $\zeta(\varphi)$ given by
\[\zeta(\varphi)=\left(\frac{\psi(Y)}{C(Y)} -\int_T^Y \frac{\psi(y)}{C^2(y)}dF^*(y),  \right)  \]
and
\[\psi(w)=\int_{w<t}[\varphi(w)-\varphi(t) ]dF(t).\]
\end{theorem}

In order to prove Theorem \ref{uclt} we first need two following propositions which their  proofs postpone in Section \ref{pr4}.
\begin{proposition}\label{w1}
Under Assumption $\bold B$ and  condition $ J(1)<\infty,$ finite dimensional distributions of $G_n$ converges to those of $W$.
\end{proposition}


\begin{remark}
Proposition \ref{w1} is correct under weak assumptions
\begin{eqnarray*}
& 1.&a_G \leq a_F \quad and \quad F\{a_F\}=0, \\
 &and&\\
 &2.& \int \frac{dF}{G}<\infty \quad  and \quad\int \frac{{\varphi}^2}{G}dF<\infty,\quad for\quad every \quad\varphi \in \mathcal{F}.
 \end{eqnarray*}
\end{remark}
In order to show the uniform CLT, we need to prove that the integral  process
\[G_n(\varphi)=\sqrt{n} \int \varphi d(F_n-F) ~~~ for ~~~\varphi\in\mathcal{F}\]
 is tight in the space of bounded functions acting on the class $\mathcal{F}$
\begin{proposition} \label{tight}
Let $J(1)<\infty.$ Then under Assumption  $\bold B,$  the process  $\{ G_n(\varphi): \varphi \in \mathcal{F} \}$  is asymptotically continuous:   for every  $\epsilon>0$,
\begin{eqnarray*}\lim_{\delta \downarrow 0} \limsup_{n\rightarrow \infty}  P^* \left\{ \sup_{d({\varphi}_1,{\varphi}_2)<\delta} |G_n({\varphi}_1)-G_n({\varphi}_2)| >\epsilon \right\} =0,
\end{eqnarray*}
where \[d({\varphi}_1,{\varphi}_2)=\left[\int({\varphi}_1-{\varphi}_2)^2dF \right]^{1/2}, \quad for ~ {\varphi}_1,{\varphi}_2 \in \mathcal{F}.\]
\end{proposition}
\begin{proof}[\textbf{Proof of Theorem \ref{uclt}}]
At first, we notice that from $J(1)<\infty,$ one can conclude the total boundedness of the metric space $(\mathcal{F}, d).$ Now, Proposition \ref{w1}, Proposition  \ref{tight} and  Pollard [\ref{pollard}, Theorem  10.2], complete the proof.
\end{proof}
\section{\bf Proofs}\label{pr4}
 Below we mention some lemmas that are used in the proofs of the main theorems. The proof of the following Lemma \ref{gha} and Lemma \ref{cltstute} appear in Stute and Wang [\ref{stut}].
\begin{lemma}\label{gha}
Let Assumption \textbf{B} is satisfied. Then under condition $J(1)<\infty,$   we have
\begin{eqnarray}
\sqrt{n} \int \varphi d(F_n-F)=n^{-1/2}\sum_{i=1}^n \zeta_i(\varphi) +\sqrt{n}R_n(\varphi) . \nonumber
\end{eqnarray}
where   \[n^{1/2} R_n(\varphi)=o_{p}(1). \]
and  $\zeta_i(\varphi)$ are i.i.d. copies of the random variables $\zeta(\varphi)$.
\end{lemma}
\begin{lemma} \label{cltstute}
 Under Assumption Lemma \ref{gha}, we have
\[\sqrt{n} \int \varphi d(F_n-F_0) \rightarrow N(0,\sigma^2)\]
with
\[\sigma^2=Var\left(\frac{\psi(X)}{C(X)} -\int_Y^X \frac{\psi(y)}{C^2(y)}dF^*(y)  \right)  \].
\end{lemma}
\begin{remark}
Lemma \ref{gha} and \ref{cltstute} are correct under weak following assumptions:\\
\textbf{1}. $a_G \leq a_F$ and $F\{a_F\}=0,$\\
\textbf{2}. $\int \frac{dF}{G}<\infty$ and $\int \frac{{\varphi}^2}{G} dF<\infty,$ for every $\varphi \in \mathcal{F}. $
\end{remark}
Write $U_n(\varphi)=n^{-1/2}\sum_{i=1}^n \zeta_i(\varphi)$ and $\mathcal{G}=\{ \zeta(\varphi):~\varphi\in \mathcal{F} \}$
\begin{proof}[\textbf{Proof of Proposition   \ref{w1}}.]
First note that by Lemma \ref{gha}, one can rewrite $G_n(\varphi)$ as
 $$G_n(\varphi)=U_n(\varphi)+\sqrt{n}R_n(\varphi),\quad for \quad \varphi \in \mathcal{F}.$$ In order to get the one dimensional central limit theorem for $G_n,$ we use Lemma \ref{cltstute} and Slutsky Theorem foe each fixed $\varphi \in \mathcal{F}.$ the  Cramer Wold device implies the result.
\end{proof}
To prove Proposition \ref{tight} we need the following Lemmas.
\begin{lemma}\label{g1}
Suppose that $J(1)<\infty,$ then we have $\int{\Phi}^2(x)dF(x)<\infty.$
\end{lemma}
\begin{proof}
Similar to the proof of Lemma 1 Bae and Kim [\ref{Bae}], from $N_{[]}(1)<\infty$ and $$\Phi(\cdot)\leq \sum_{i=0}^{N_{[]}(1)}\left(|l_{i}(\cdot)|+|u_{i}(\cdot)|\right),$$ we can conclude the square integrability of $\Phi$.
\end{proof}
\begin{lemma} \label{dav}
Suppose $J(1)<\infty$ and $a_G <a_F,$ then
\begin{eqnarray*}
\int\frac{dF}{G}<\infty \quad and \quad
 \int \frac{{\Phi}^2}{G}dF<\infty.
\end{eqnarray*}
\end{lemma}
\begin{proof}
From $J(1)<\infty$ and Lemma \ref{g1}, follows $\int {\Phi(x)}^2 dF(x)<\infty.$  From this fact and $a_G <a_F$, one can write $\int \frac{dF}{G}<\infty$ and $\int \frac{{\Phi}^2}{G} dF<\infty.$
\end{proof}
\begin{lemma} \label{tightR}
Under Assumption $\bold B$ and condition  $J(1)<\infty
 ,$ for every $ \epsilon>0$  we have
\[ \limsup_{n\rightarrow\infty} P^*(\sqrt{n} ||R_n||_{\mathcal{F}} >\epsilon ) =0  \]
that is
$\{ \sqrt{n}R_n(\varphi) : \varphi \in \mathcal{F} \}$
is tight.
\end{lemma}
\begin{proof}
From Remark $1.1$ of Stute and Wang [\ref{stut}] and Lemma \ref{dav}, one can write  ${||R_{n}||}_{\mathcal{F}}=o_{p}(n^{-1/2}).$
\end{proof}
\begin{lemma} \label{kol}
Let $J(1)<\infty,$ then under Assumption $\bold B$
\begin{eqnarray*}
\int_{0}^{1} {\left[\log N_{[]} (\epsilon,\mathcal{G},d)\right]}^{1/2}d\epsilon<\infty
\end{eqnarray*}
holds, thus Theorem 3.1 and Theorem 3.3. of Ossiander [\ref{Oss}] can be applied to the process $\{U_n(\varphi):\varphi \in \mathcal{F}\}.$
\end{lemma}
\begin{proof}
This result is an easy consequence of Jensen and $c_p$ inequalities. Fix  $\epsilon>0.$ By definition of $N_{[]}(\epsilon),$ there exists $$\left\{\left[l_{0},{u}_{0}\right],\ldots,[l_{N_{[]}(\epsilon)},{u}_{N_{[]}(\epsilon)}]\right\}$$ so that for every $\varphi \in \mathcal{F}$ there exists
$0 \leq i\leq N_{[]}(\epsilon)$ satisfying ${l}_{i}\leq \varphi \leq {u}_{i}$ and $d({l}_{i},{u}_{i})<\epsilon.$ Let $g\in \mathcal{G}.$ Then $g=\zeta(\varphi)$ for some $\varphi\in \mathcal{F}.$ Now define the brackets for the class $\mathcal{G}$ by the equations
\begin{eqnarray*}
g^{l}_{j}:=\frac{\int_{Y<t}\left({l}_{j}(Y)-{u}_{j}(t)\right)d F(t)}{C(Y)}-\int_{T}^{Y}\frac{\int_{y<t}\left({u}_{j}(y)-{l}_{j}(t)\right)d F(t)}{C^{2}(y)} dF^{*}(y)
\end{eqnarray*}
and
\begin{eqnarray*}
g^{u}_{j}:=\frac{\int_{Y<t}\left({u}_{j}(Y)-{l}_{j}(t)\right)d F(t)}{C(Y)}-\int_{T}^{Y}\frac{\int_{y<t}\left({l}_{j}(y)-{u}_{j}(t)\right)d F(t)}{C^{2}(y)} dF^{*}(y),
\end{eqnarray*}
for $j=0,\ldots,N_{[]}(\epsilon).$ Simplify the notations by writing ${l}={l}_{j}, {u}={u}_{j}, g^{l}=g^{l}_{j}$ and $g^{u}=g^{u}_{j}.$ Obviously, we have $g^{l}\leq g\leq g^{u}.$ Using $c_p$ inequality and Jensen inequality, we have
\begin{eqnarray*}
d^{2}(g^{l},g^{u})&=& \int{(g^{u}-g^{l})}^{2}dF\\
&=&\int\Big(\int_{Y<t}\frac{\big[{u}(Y)-{l}(Y)+{u}(t)-{l}(t)\big]}{C(Y)}dF(t)\\
&&+\int_{T}^{Y}\int_{y<t}\frac{\left[({u}(y)-{l}(y))+({u}(t)-{l}(t))\right]}{C^{2}(y)}dF(t)dF^{*}(y){\Big)}^{2} dF\\
&\leq& C\int{\Big[\int_{Y<t}\frac{\big[{u}(Y)-{l}(Y)+{u}(t)-{l}(t)\big]}{C(Y)}dF(t)\Big]}^{2}dF\\
&&+ C\int{\Big[\int_{T}^{Y}\int_{y<t}\frac{\big[{u}(y)-{l}(y)+{u}(t)-{l}(t)\big]}{C^{2}(y)}dF(t)dF^{*}(y)\Big]}^{2}dF\\
&\leq& C\int\Big[\int\frac{{\big({u}(Y)-{l}(Y)+{u}(t)-{l}(t)\big)}^{2}}{C^{2}(Y)}dF(t)\Big]dF\\
&& +C \int \Big[\int_{T}^{Y}\frac{{u}(y)-{l}(y)}{C^{2}(y)}dF^{*}(y)\\
&+&\int_{T}^{Y}\frac{1}{C^{2}(y)} \left( \int( {u}(t)-{l}(t))d F(t)\right)d F^{*}(y){\Big]}^{2}dF\\
&\leq& C\int\Big[\int\frac{{\big({u}(Y)-{l}(Y)+{u}(t)-{l}(t)\big)}^{2}}{C^{2}(Y)}dF(t)\Big]dF\\
&\quad+& C\int{\left[\int_{0}^{\infty}\frac{{u}(y)-{l}(y)}{C^{2}(y)}dF^{*}(y)\right]}^{2}dF\\
\end{eqnarray*}
\begin{eqnarray*}
&\quad+& C\int{\left[\int_{0}^{\infty}\frac{1}{C^2(y)}\left(\int\left(u(t)-l(t)\right)dF(t)\right)dF^{*}(y)\right]}^{2}dF\\
&\leq&C\int\Big[\int\frac{{\big({u}(Y)-{l}(Y)+{u}(t)-{l}(t)\big)}^{2}}{C^{2}(Y)}dF(t)\Big]dF\\
&\quad+&C\int\int_{0}^{\infty}\frac{{\left(u(y)-l(y)\right)}^2}{C^4(y)}dF^{*}(y)dF\\
&\quad+&C\int \int_{0}^{\infty}\frac{1}{C^4(y)}{\left(\int\left(u(t)-l(t)\right)dF(t)\right)}^{2}dF^{*}(y)dp\\
&\leq&C\int\Big[\int\frac{{\big({u}(Y)-{l}(Y)+{u}(t)-{l}(t)\big)}^{2}}{C^{2}(Y)}dF(t)\Big]dF\\
&\quad+&C\int\int_{0}^{\infty}\frac{{\left(u(y)-l(y)\right)}^2}{C^4(y)}dF^{*}(y)dF\\
&\quad+&C\int \int_{0}^{\infty}\frac{1}{C^4(y)}\left(\int{\left(u(t)-l(t)\right)}^{2}dF(t)\right)dF^{*}(y)dF\\
&=:&A_1+A_2+A_3.
\end{eqnarray*}
 The function  $C(\cdot)$ is strictly positive on  $a_G <y< b_F$, thus $$\left( \inf_{a_G <y< b_F} C(y) \right)^{-1}>0.$$
Using $c_p$ inequality, we see that
\begin{eqnarray*}
A_1&\leq&C\int\int\frac{{\left(u(Y)-l(Y)\right)}^2}{C^2(Y)}dF(t)dF\\
&\quad+&C\int\int\frac{{\left(u(t)-l(t)\right)}^2}{C^2(Y)}dF(t)dF\\
&\leq& C\left( \inf_{a_G <y< b_F} C(y) \right)^{-1}d^2(u,l).
\end{eqnarray*}
Also
\begin{eqnarray*}
A_2
&\leq & C\left( \inf_{a_G <y< b_F} C(y) \right)^{-4}d^2(u,l).
\end{eqnarray*}
For the third term $C$,
\begin{eqnarray*}
A_3
&\leq& C\left( \inf_{a_G <y< b_F} C(y) \right)^{-4}d^2(u,l).
\end{eqnarray*}
Hence,
\begin{eqnarray*}
d^2(g^l,g^u)\leq C d^2(l,u).
\end{eqnarray*}
Finally the result follows from condition $J(1)<\infty.$
\end{proof}
\begin{proof}[\textbf{Proof of Proposition \ref{tight}}.]
Notice that,
\[ |G_n({\varphi}_1) - G_n({\varphi}_2) |\leq |U_n({\varphi}_1)-U_n({\varphi}_2) | +2\sqrt{n} || R_n ||_{\mathcal{F}} \quad a.s. \]
Now from Lemma \ref{kol} and Lemma \ref{tightR} we can see
\begin{eqnarray*}
P^*\left\{{||G_n||}_{\delta}>3\epsilon\right\}\leq  P^*\left\{{||U_n||}_{\delta}>\epsilon\right\}+ P^*\left\{{2n^{1/2}||R_n||}_{\mathcal{F}}>2\epsilon\right\}<3\epsilon.
\end{eqnarray*}
eventually. This complete the proof.
\end{proof}

\end{document}